\documentclass{article}
\usepackage{arxiv}

\usepackage[utf8]{inputenc} 
\usepackage[T1]{fontenc}    
\usepackage{hyperref}       
\usepackage{url}            
\usepackage{booktabs}       
\usepackage{amsfonts}       
\usepackage{nicefrac}       
\usepackage{microtype}      

\usepackage{amsthm,amsmath}
\usepackage{array}

\usepackage{float}

\usepackage{tikz}
\usepackage{binarytree}
\usetikzlibrary{graphs}

\title{Hopf algebras on planar trees and permutations}
\date{}

\author{
	Diego Arcis \\
	Facultad de Ciencias de la Salud\\
  	Universidad Aut\'onoma de Chile -- Sede Talca\\
  	5 Poniente 1670 -- 3460000 Talca, Chile \\
  	\texttt{diego.arcis@uautonoma.cl} \\
    \And
  	Sebasti\'an M\'arquez \\
	Facultad de Ingenier\'ia\\
  	Universidad Aut\'onoma de Chile -- Sede Talca\\
  	4 Norte 99 -- 3460000 Talca, Chile \\
  	\texttt{sebastian.marquez@uautonoma.cl} \\
}

\begin{document}

\newtheorem{thm}{Theorem}[section]
\newtheorem{crl}[thm]{Corollary}
\newtheorem{dfn}[thm]{Definition}
\newtheorem{lem}[thm]{Lemma}
\newtheorem{pro}[thm]{Proposition}

\theoremstyle{definition}
\newtheorem{rem}{Remark}[section]
\newtheorem{exm}[thm]{Example}

\newcommand\N{\mathbb{N}}

\newcommand\T{\mathcal{T}}
\newcommand\PP{\mathcal{P}}

\newcommand\Sym{\mathfrak{S}}

\newcommand\s{\text{s}}

\newcommand\id{\mathrm{id}}

\newcommand\ind{\mathrm{ind}}
\newcommand\ord{\mathrm{ord}}

\newcommand\Prim{\mathrm{Prim}}

\newcommand\field{K}
\newcommand\tprod{*}

\def\blue{\color{blue}}
\def\red{\color{red}}

\newcommand{\hash}[1]{\,\,\#_{#1}\,\,}

\maketitle

\newcommand\treeroot{
	\includegraphics{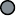} 
}

\newcommand\treeone{
	\includegraphics{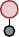} 
}

\newcommand\ktreeone{
	\includegraphics{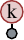} 
}

\newcommand\binarytreeone{
	\raisebox{-.25\height}{\includegraphics{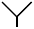}} 
}

\newcommand{\figureone}[1]{
	\begin{array}{ccccc}\includegraphics{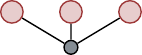}&\qquad\includegraphics{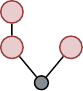}&\qquad\includegraphics{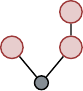}&\qquad\includegraphics{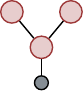}&\qquad\includegraphics{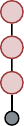}\end{array} 
}

\newcommand{\figuretwo}[1]{
	\begin{array}{cccc}\includegraphics{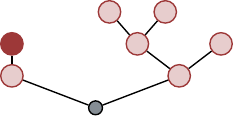}&\quad\includegraphics{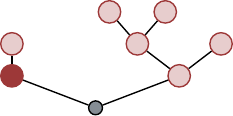}&\quad\includegraphics{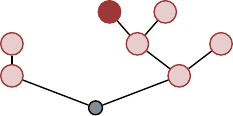}&\quad\includegraphics{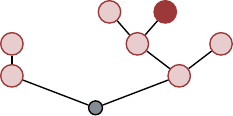}\\[5mm]\includegraphics{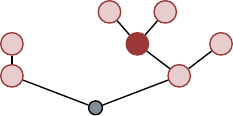}&\quad\includegraphics{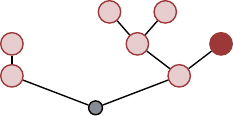}&\quad\includegraphics{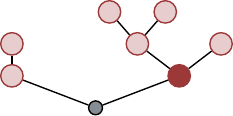}&\quad\includegraphics{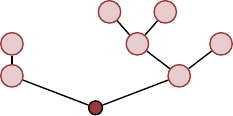}\end{array} 
}

\newcommand{\figurethr}[1]{
	\includegraphics{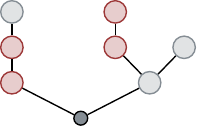}\includegraphics{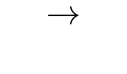}\includegraphics{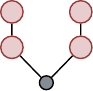} 
}

\newcommand{\figurefou}[1]{
	\begin{array}{ccccccc}\includegraphics{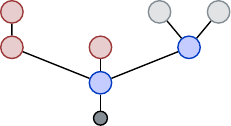}&\includegraphics{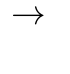}&\includegraphics{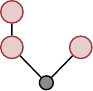}&,&\includegraphics{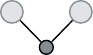}&,&\includegraphics{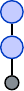}\end{array} 
}

\newcommand{\figurefiv}[1]{
	\begin{array}{c}\Delta(\includegraphics{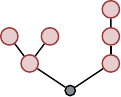})\quad=\quad\includegraphics{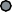}\otimes\includegraphics{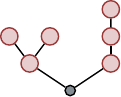}\quad+\quad\includegraphics{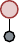}\otimes\includegraphics{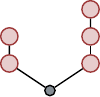}\quad+\quad\includegraphics{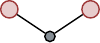}\otimes\includegraphics{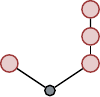}\\[0.3cm]+\quad\includegraphics{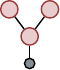}\otimes\includegraphics{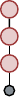}\quad+\quad\includegraphics{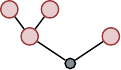}\otimes\includegraphics{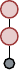}\quad+\quad\includegraphics{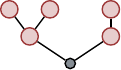}\otimes\includegraphics{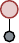}\quad+\quad\includegraphics{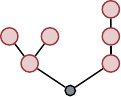}\otimes\includegraphics{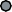}\,.\end{array} 
}

\newcommand{\figuresix}[1]{
	\includegraphics{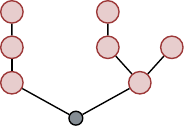}\includegraphics{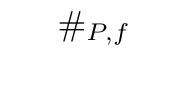}\includegraphics{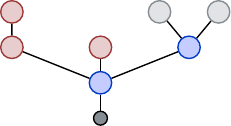}\includegraphics{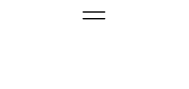}\includegraphics{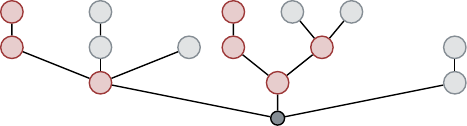} 
}

\newcommand{\figuresev}[1]{
	\begin{array}{c}\includegraphics{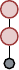}*\includegraphics{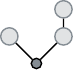}\quad=\quad\includegraphics{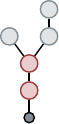}\quad+\quad\includegraphics{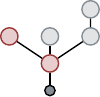}\quad+\quad\includegraphics{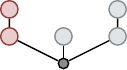}\quad+\quad\includegraphics{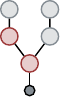}\quad+\quad\includegraphics{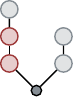}\\[2mm]\quad+\quad\includegraphics{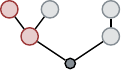}\quad+\quad\includegraphics{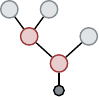}\quad+\quad\includegraphics{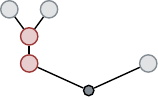}\quad+\quad\includegraphics{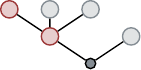}\quad+\quad\includegraphics{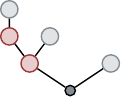}\end{array} 
}

\newcommand{\figureeig}[1]{
	\begin{array}{c}\includegraphics{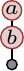}*\includegraphics{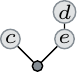}\quad=\quad\includegraphics{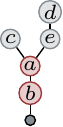}\quad+\quad\includegraphics{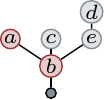}\quad+\quad\includegraphics{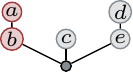}\quad+\quad\includegraphics{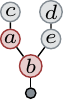}\quad+\quad\includegraphics{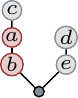}\\[2mm]\quad+\quad\includegraphics{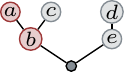}\quad+\quad\includegraphics{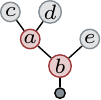}\quad+\quad\includegraphics{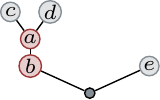}\quad+\quad\includegraphics{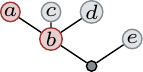}\quad+\quad\includegraphics{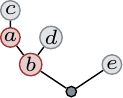}\end{array} 
}

\newcommand{\figurenin}[1]{
	\begin{array}{c}\Delta(\includegraphics{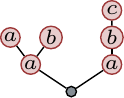})\quad=\quad\includegraphics{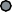}\otimes\includegraphics{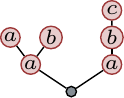}\quad+\quad\includegraphics{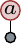}\otimes\includegraphics{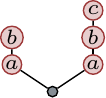}\quad+\quad\includegraphics{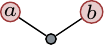}\otimes\includegraphics{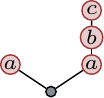}\\[0.3cm]+\quad\includegraphics{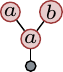}\otimes\includegraphics{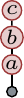}\quad+\quad\includegraphics{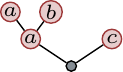}\otimes\includegraphics{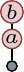}\quad+\quad\includegraphics{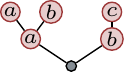}\otimes\includegraphics{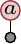}\quad+\quad\includegraphics{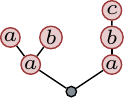}\otimes\includegraphics{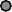}\,.\end{array} 
}

\newcommand{\figureten}[1]{
	\includegraphics{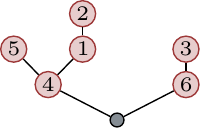}\includegraphics{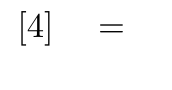}\includegraphics{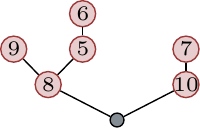} 
}

\newcommand{\figureele}[1]{
	\begin{array}{c}\includegraphics{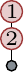}\star\includegraphics{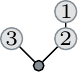}\quad=\quad\includegraphics{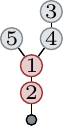}\quad+\quad\includegraphics{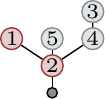}\quad+\quad\includegraphics{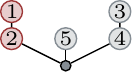}\quad+\quad\includegraphics{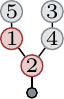}\quad+\quad\includegraphics{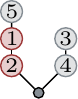}\\[2mm]\quad+\quad\includegraphics{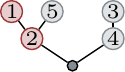}\quad+\quad\includegraphics{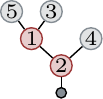}\quad+\quad\includegraphics{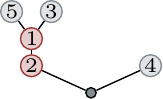}\quad+\quad\includegraphics{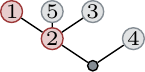}\quad+\quad\includegraphics{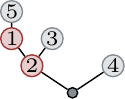}\end{array} 
}

\newcommand{\figuretwe}[1]{
	\includegraphics{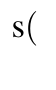}\includegraphics{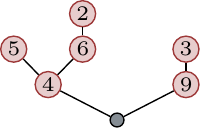}\includegraphics{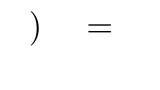}\includegraphics{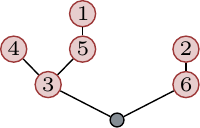} 
}

\newcommand{\figuretenthr}[1]{
	\begin{array}{c}\Delta(\includegraphics{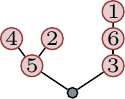})\quad=\quad\includegraphics{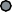}\otimes\includegraphics{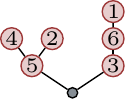}\quad+\quad\includegraphics{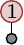}\otimes\includegraphics{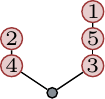}\quad+\quad\includegraphics{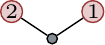}\otimes\includegraphics{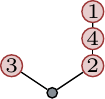}\\[0.3cm]+\quad\includegraphics{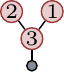}\otimes\includegraphics{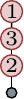}\quad+\quad\includegraphics{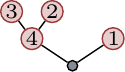}\otimes\includegraphics{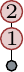}\quad+\quad\includegraphics{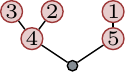}\otimes\includegraphics{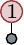}\quad+\quad\includegraphics{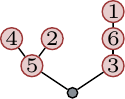}\otimes\includegraphics{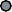}\,.\end{array} 
}

\newcommand{\figuretenfou}[1]{
	\begin{array}{cccccccccccc}\includegraphics{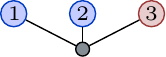}&\quad&\includegraphics{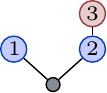}&\quad&\includegraphics{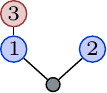}&\quad&\includegraphics{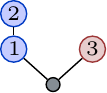}&\quad&\includegraphics{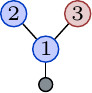}&\quad&\includegraphics{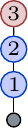}\end{array} 
}

\newcommand{\figuretenfiv}[1]{
	\includegraphics{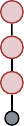}\includegraphics{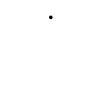}\includegraphics{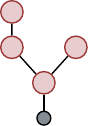}\includegraphics{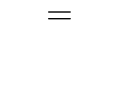}\includegraphics{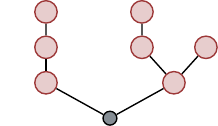} 
}

\newcommand{\figuretensix}[1]{
	\includegraphics{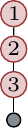}\includegraphics{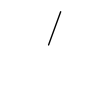}\includegraphics{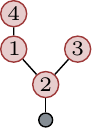}\includegraphics{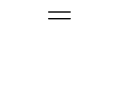}\includegraphics{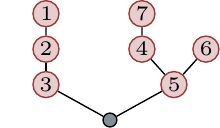} 
}

\newcommand{\figuretensev}[1]{
	\includegraphics{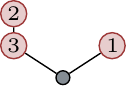} 
}

\newcommand{\figureteneig}[1]{
	\begin{array}{ccl}\includegraphics{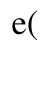}\includegraphics{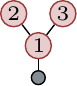}\includegraphics{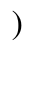}&\includegraphics{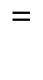}&\includegraphics{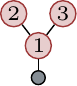}\quad\includegraphics{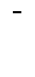}\quad\includegraphics{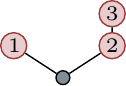}\\[0.1cm]\includegraphics{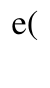}\includegraphics{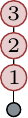}\includegraphics{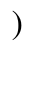}\quad&\includegraphics{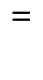}&\quad\includegraphics{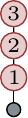}\quad\includegraphics{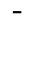}\quad\includegraphics{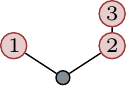}\quad\includegraphics{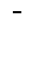}\quad\includegraphics{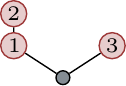}\quad\includegraphics{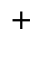}\quad\includegraphics{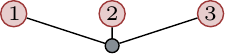}\\[0.1cm]\includegraphics{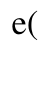}\includegraphics{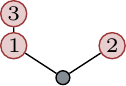}\includegraphics{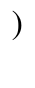}&\includegraphics{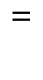}&\includegraphics{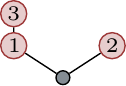}\quad\includegraphics{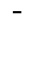}\quad\includegraphics{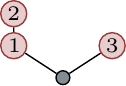}\end{array} 
}

\newcommand{\figuretennin}[1]{
	\includegraphics{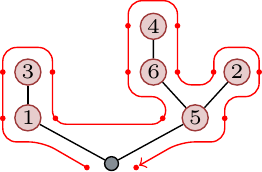} 
}

\newcommand{\figuretwenty}[1]{
	\includegraphics{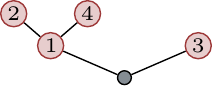} 
}

\newcommand{\figuretweone}[1]{
	\includegraphics{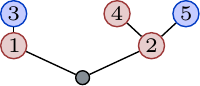} 
}

\newcommand{\figuretwetwo}[1]{
	\includegraphics{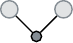}\cdot_{\rm d}\includegraphics{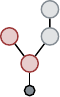}\quad=\quad\includegraphics{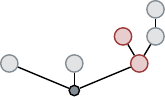}\quad\!\!\!+\!\!\!\quad\includegraphics{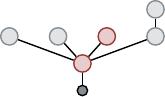}\quad\!\!\!+\!\!\!\quad\includegraphics{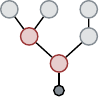}\quad\!\!\!+\!\!\!\quad\includegraphics{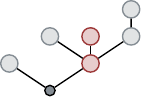}\quad\!\!\!+\!\!\!\quad\includegraphics{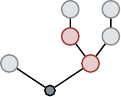}\quad\!\!\!+\!\!\!\quad\includegraphics{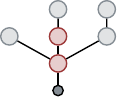} 
}

\newcommand\figuretwethr{
	Y_0=\left\{\,\raisebox{-.35\height}{\includegraphics{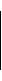}}\,\right\},\quad Y_1=\left\{\raisebox{-.3\totalheight}{\includegraphics{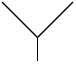}}\right\},\quad Y_2=\left\{\raisebox{-.3\totalheight}{\includegraphics{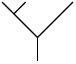}},\,\,\raisebox{-.3\totalheight}{\includegraphics{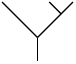}}\right\} 
}

\newcommand\figuretwefou{
	x\,\,=\raisebox{-.35\height}{\includegraphics{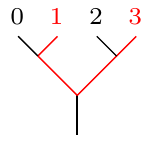}}\qquad\quad x_{[1,3]}\,\,=\raisebox{-.35\height}{\includegraphics{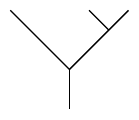}} 
}

\newcommand\figuretwefiv[1]{
	\varphi\left(\,\raisebox{-.4\height}{\includegraphics{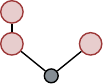}}\right)\,\,=\,\,
	\raisebox{-.4\height}{\includegraphics{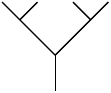}} 
}

\newcommand\figuretwesix[1]{
	\raisebox{-.4\height}{\includegraphics{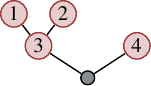}}\quad\,\,\stackrel{^{\varphi}}{\leftrightarrow}\quad\raisebox{-.4\height}{\includegraphics{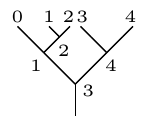}}
}

\begin{abstract}
We endow the space of rooted planar trees with an structure of Hopf algebra. We prove that variations of such a structure lead to Hopf algebras on the spaces of labelled trees, $n$--trees, increasing planar trees and sorted trees. These structures are used to construct Hopf algebras on different types of permutations. In particular, we obtain new characterizations of the Hopf algebras of Malvenuto--Reutenauer and Loday--Ronco via planar rooted trees.
\end{abstract}

\keywords{Hopf algebras \and Infinitesimal bialgebras \and Trees \and Permutations}

\section*{Introduction}

In 1989, Grossman and Larson introduced a Hopf algebra structure on the vector space of rooted trees \cite{GrLa89}. This idea appeared due to a Grossman previous work on data structures in which the trees could be operated. They not only prove that this space can be given with a Hopf algebra structure, but also the spaces of labelled rooted trees, ordered rooted trees, labelled ordered rooted trees, heap ordered rooted trees and labelled heap ordered rooted trees \cite{GrLa89}.

Since the seminal work of Grossman and Larson mentioned above, several other Hopf algebras spanned by trees have been defined, for example the Hopf algebra of rooted trees introduced by Connes and Kreimer in \cite{CoKr98} and its non commutative version for planar rooted trees introduced simultaneously by Foissy and Holtkamp in \cite{Fo02,Ho03}. On the other hand, Loday and Ronco introduced in \cite{LoRo98} a Hopf algebra spanned by planar binary trees.

It was in 1995, in order to study relations between quasi-symmetric functions and the Solomon descent algebra, that Malvenuto and Reutenauer defined a Hopf algebra on the space $\field[\Sym_{\infty}]$ spanned by the collection of all permutations \cite{MaRe95}. This Hopf algebra can be regarded as a structure of trees as well, because every permutation can be released as a labelled planar binary tree. In \cite{AgSo05}, it was proved that the Grossman--Larson Hopf algebras of ordered trees and heap ordered trees are dual to the associated graded Hopf algebras of planar binary trees and permutations respectively.

In \cite{Fo09}, Foissy introduced a different coproduct on the space spanned by planar rooted trees, which is infinitesimal, in the sense of \cite{LoRo06}, respect to the algebraic structure considered in \cite{Fo02,Ho03}. This coproduct was described in \cite{Mq18} by using convex trees that are obtained via the depth-first post order traverse of trees. It is a natural problem to endow this coproduct with an associative product that gives the space of planar rooted trees with a structure of Hopf algebra and so study its adaptation to other related trees with additional structures as it was done in \cite{GrLa89}. In the present paper, we deal with this problem. In particular, we show that the Hopf algebras of permutations and of planar binary trees can be performed via planar rooted trees through of our construction. Additionally, we study how to realise these structures as Hopf algebras of different types of permutations, namely classic permutations, Stirling permutations, etc.

The paper is organized as follows:

In Section \ref{027}, we give the basic definitions and notations about planar rooted trees. In particular, we recall the depth-first post order traverse of a tree and introduce partitions of convex trees respect to it. In Section \ref{028}, we recall the coproduct $\Delta$ defined in \cite{Mq18} and its infinitesimal structure on the space of planar rooted trees $\field[\T_{\infty}]$. We characterize the dual product induced by $\Delta$ in terms of ordered partitions of irreducible trees and order preserving maps. In Section \ref{029} we introduce an associative product on the space $\field[\T_{\infty}]$ so that it is a $2$-associative Hopf algebra (Theorem \ref{009}) in the sense of \cite{LoRo06}. We also show that the collection of $X$--labelled trees $\T(X)$ inherits directly the Hopf structure of $\field[\T_{\infty}]$.

In Section \ref{031} it is showed that the Hopf algebra constructed in Section \ref{029} can be adapted on spaces of trees with additional structures. We prove that such a structure can be adapted for the space of $n$--trees $\field[T[\infty]]$ by applying standardization on its convex trees. Also, it is showed that the space $\field[I[\infty]]$ spanned by increasing trees is a sub Hopf algebra of $\field[\T[\infty]]$. We conclude the section by proving that the collection of sorted trees can also be given with a Hopf structure which is regarded as a Hopf algebra of non planar increasing trees.

Section \ref{032} deals with the problem of construct Hopf algebras on permutations via the behaviour of the planar rooted trees in $\T[\infty]$ and its related structures. In particular, we introduce the collection of treed permutations and show that $\field[\T[\infty]]$ induces a Hopf algebra structure on the vector space generated by these permutations. Also it is proved that the set of Stirling permutations spans a sub Hopf algebra of it. We finish the section by showing that the Hopf algebra of permutations of Malvenuto--Reutenauer \cite{MaRe95,AgSo02} can be released through the Hopf structure of sorted trees.

To finish, in Section \ref{033} we use an explicit bijection between planar trees and planar binary trees, described in \cite{AgSo04}, to show that the Loday--Ronco Hopf algebra can be characterized through the dual structure defined on $\field[\T_{\infty}]$ in Section \ref{029}.

We assume the reader is familiar with the terminology of Hopf algebras, see for example \cite{Ab04} for more details.

In what follows of the paper, an \emph{algebra} means an associative $\field$--algebra with unit and a \emph{coalgebra} means a coassociative $\field$--coalgebra with counit, where $\field$ is a field.

\section{Planar rooted trees}\label{027}

Here we give the basic definitions and notations about planar rooted trees.

In what follows of this paper, a \emph{tree} will be a planar rooted tree.

The \emph{degree} of a tree $t$, denoted by $|t|$, is its number of non root nodes. We denote by $\treeroot$ the unique tree of degree zero.

It is well known that for every integer $n\geq0$ the number of all trees of degree $n$ is the $n$th Catalan number $c_n=\frac{1}{n+1}\binom{2n}{n}$. We will denote by $\T_n$ the set of all trees of degree $n$ and by $\T_{\infty}$ the set of all trees, hence\[\T_{\infty}=\bigsqcup_{n\geq0}\T_n.\]For instance, for $n=3$ we have the following five trees:\begin{figure}[H]
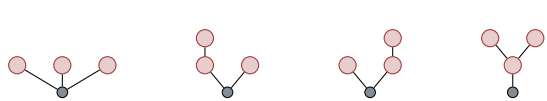
$\figureone{#1}$\centering\caption{Trees of degree $3$.}\end{figure}

Every tree induces an order on its nodes given by the depth-first post order traverse, that is, for each node we first traverse its subtrees from left to right. Note that under this order, the root is always maximal.

For instance, we have.\begin{figure}[H]
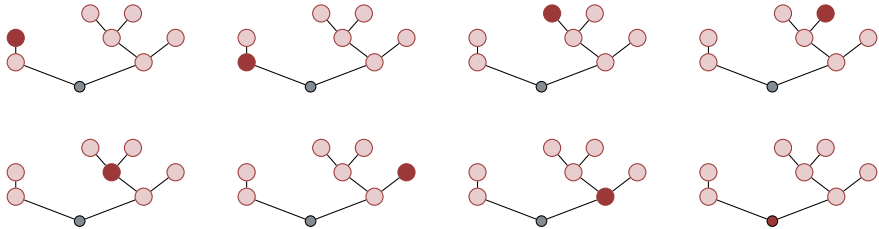
$\figuretwo{#1}$\centering\caption{Traversing of a tree of degree $7$.}\end{figure}

For a tree $t$ we will denote by $N(t)$ its set of nodes and set $N^*(t)=N(t)\backslash\{\text{root}\}$.

Given a tree $t$ and a set $A$ of nodes that does not contain the root of $t$, we denote by $t_A$ the tree obtained by removing the nodes out of $A$ and so contract its edges. Note that $t_{[]}=\treeroot$ and $t_{N(t)}=t$. We call $t_A$ a \emph{convex tree} of $t$ if $A\cap N(t)$ is an interval of $N^*(t)$.

For instance, below is a tree $t$ of degree $7$ and the convex tree $t_I$ of degree $4$ obtained by removing the first one and the last two nodes from the original tree $t$.\begin{figure}[H]
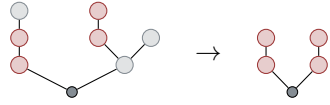
$\figurethr{#1}$\centering\caption{Convex tree obtained from a tree of bigger degree.}\label{004}\end{figure}

Given a tree $t$ of degree $n$ with nodes $u_1,\ldots,u_n$, we set\begin{equation}\label{006}t_{[i,j]}=\left\{\begin{array}{ll}
t_{[u_h,u_k]}&\text{if }i\leq j\text{ and }[i,j]\cap[n]=[h,k]\neq\emptyset\\t_{[]}&\text{otherwise}.\end{array}\right.\end{equation}With this notation, the tree obtained in Figure \ref{004} is $t_{[2,5]}$. Note that $t_{[-2,9]}=t$ and $t_{[8,10]}=\treeroot$.

A \emph{$k$--partition} of a linearly ordered set $(A,<)$ is a collection of nonempty intervals $I_1,\ldots,I_k$ of $A$ such that $A=I_1\sqcup\cdots\sqcup I_k$ and $a<b$ for all $(a,b)\in I_i\times I_j$ with $i<j$. A \emph{$k$--partition} of a tree $t$ is an ordered collection $P=(t_1,\ldots,t_k)$ of convex trees of $t$ called \emph{blocks} of $P$, where $t_i=t_{I_i}$ for some $k$--partition $(I_1,\ldots,I_k)$ of $N^*(t)$. More specifically, by using Equation \ref{006}, a $k$--partition of a tree $t$ is a collection of the form\begin{equation}\label{007}\{t_{[1,n_1]},t_{[n_1+1,n_2]},\ldots,t_{[n_{k-1}+1,n]}\},\quad 1<n_1<\cdots<n_{k-1}<n.\end{equation}For instance, below is a $3$--partition $\{t_{[1,3]},t_{[4,5]},t_{[6,7]}\}$ for a tree $t$ of degree $7$.\begin{figure}[H]
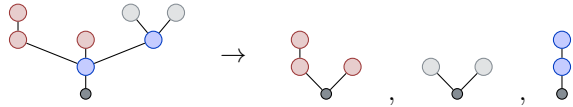
$\figurefou{#1}$\centering\caption{$3$--partition of a tree.}\label{005}\end{figure}Note that a tree $t$ of degree $n$ admits $k$--partitions with $k\leq n$. There is a unique $1$--partition containing itself and a unique $n$--partition formed by $n$ trees of degree $1$. We will denote by $\PP_k(t)$ the set of all $k$--partitions of $t$, and set\[\PP(t)=\PP_1(t)\sqcup\cdots\sqcup\PP_n(t).\]

\begin{pro}\label{008}
The number of $k$--partitions of a tree $t$ of degree $n$ is $|\PP_k(t)|=\binom{n-1}{k-1}$. Hence $|\PP(t)|=2^{n-1}$.
\end{pro}
\begin{proof}
From Equation \ref{007} we obtain that $|\PP_k(t)|$ corresponds to the number of $k$--compositions of $n$, which is well known to be $\binom{n-1}{k-1}$. Hence $|\PP(t)|=\sum_{k=1}^n\binom{n-1}{k-1}=2^{n-1}$.
\end{proof}

\section{Infinitesimal structure}\label{028}

Here we recall the structure of \emph{infinitesimal bialgebra} of the vector space spanned by trees defined in \cite{Fo09,Mq18}.

Recall that an \emph{infinitesimal bialgebra}, in the sense of \cite{Lo04,LoRo06}, is an algebra $H$ equipped with a coproduct $\Delta:H\to H\otimes H$ satisfying\[\Delta(xy)=\Delta(x)(1\otimes y)+(x\otimes1)\Delta(y)-x\otimes y\,\,\text{ for all }\,\,x,y\in H.\]

Given a vector space $V$, it is well known that the tensor module $T(V)=\bigoplus_{n\geq0}V^{\otimes n}$ with the concatenation product is an infinitesimal bialgebra with the \emph{deconcatenation} coproduct defined as follows\[\Delta(v_1\cdots v_n)=\sum_{i=0}^n(v_1\cdots v_i)\otimes (v_{i+1}\cdots v_n).\]

\begin{thm}[Loday--Ronco {\cite[Theorem 2.6]{LoRo06}}]\label{026}
Every connected infinitesimal bialgebra $H$ is isomorphic to $T(\Prim(H))$ with the infinitesimal bialgebra structure defined above.
\end{thm}

Recall that $\field[\T_{\infty}]$ denotes the vector space over $\field$ spanned by $\T_{\infty}$ which admits a natural graduation given as follows\[\field[\T_{\infty}]=\bigoplus_{n\geq0}\field[\T_n].\]

There is a natural product with unit $\treeroot$ that endows $K[\T_{\infty}]$ with an structure of algebra. Indeed, given two trees $t,w\in\T_{\infty}$, the \emph{dot product} of $t$ with $w$, denoted by $t\cdot w$, is the tree obtained by identifying the roots of $t$ and $w$. For instance\begin{figure}[H]
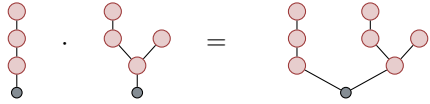
$\figuretenfiv{#1}$\centering\caption{Dot product of trees.}\end{figure}

We say that $t\in\T_{\infty}$ is \emph{reducible}, respect to the dot product, if there are non trivial trees $u,v\in\T_{\infty}$ such that $t=u\cdot v$, otherwise $t$ is called \emph{irreducible}. Note that a tree is irreducible if its root has a unique child, and that every non trivial tree can be uniquely written as a product of irreducible trees.

The following coproduct $\Delta:\field[T_{\infty}]\to\field[T_{\infty}]\otimes\field[T_{\infty}]$ was originally introduced in \cite{Fo09}, described in terms of forests of trees, and appears naturally in the notion of compatible associative bialgebra defined in \cite{Mq18}. Here we use $\Delta$ as described in \cite{Mq18} in terms of convex trees that are obtained via the depth-first post traverse of trees, that is\[\Delta(t):=\sum_{k=0}^nt_{[1,k]}\otimes t_{[k+1,n]},\qquad t\in\T_{\infty},\,n=|t|.\]We will denote the coproduct above, by using Sweedler's notation, simply by $t_{(1)}\otimes t_{(2)}$.

\begin{exm}
Below the coproduct of a tree of degree six.\begin{figure}[H]
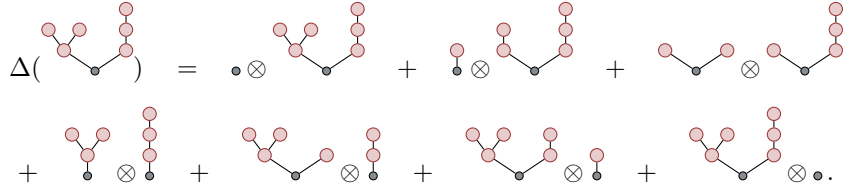
$\figurefiv{#1}$\centering\caption{Coproduct of a tree.}\end{figure}
\end{exm}

We also may endow $\field[\T_{\infty}]$ with a counit $\epsilon:\field[\T_{\infty}]\to\field$ defined by $\epsilon(t)=1$ if $|t|=0$ and $\epsilon(t)=0$ otherwise.

\begin{pro}[Foissy {\cite[Theorem 9]{Fo09}}, M\'arquez {\cite[Proposition 4.8]{Mq18}}]\label{010}
The tuple $(\field[\T_{\infty}],\cdot,\Delta)$ is a graded infinitesimal bialgebra.
\end{pro}

\begin{rem}
Since $\field[\T_{\infty}]$ is connected, Proposition \ref{010} and Theorem \ref{026} imply that $\field[\T_{\infty}]$ is isomorphic to the infinitesimal bialgebra $T(\Prim(\field[\T_{\infty}]))$. In particular, $\field[\T_{\infty}]$ is cofree among the collection of connected coalgebras.
\end{rem}

\subsection{Dual structure}

For a graded infinitesimal bialgebra $H$ with finite dimensional homogeneous components, we have that its dual graded space $H^*$ is an infinitesimal bialgebra.

The dual infinitesimal structure of $\field[\T_{\infty}]$ was studied by Foissy in \cite{Fo09}. Here we will give a description of the dual product in terms of ordered partitions of irreducible trees and order preserving maps. Let $\cdot_{\rm d}$ and $\Delta_{\rm d}$ be the respective dual product and dual coproduct induced by the infinitesimal bialgebra $(\field[\T_{\infty}],\cdot,\Delta)$. By identifying each tree $t$ with $\delta_t$ we have that $t\cdot_{\rm d}w=\sum u$, where $t\otimes w$ is an addend of $\Delta(u)$, and $\Delta_{\rm d}(t)=\sum_{u\cdot w=t}u\otimes w$. 

It is immediate that the coproduct may be expressed as\[\Delta_{{\rm d}}(t_1\cdots t_k)\sum_{i=0}^{k}t_1\cdots t_i\otimes t_{i+1}\cdots t_k,\]where $t=t_1\cdots t_k$ is the unique decomposition of $t$ as irreducible trees.

For the product, given a tree $t$ and a leaf of it, there is a unique path from the root of $t$ to this leaf called \emph{branch}. We denote by $p_1(t)$ the branch of $t$ correspondent to the leftmost leaf of $t$. We consider the set of nodes of $p_1(t)$, which includes the root, ordered from bottom to top.

Let $t,w$ be two trees and consider $P=(u_1,\ldots,u_j)$ be an ordered partition of the unique decomposition of $t$ in irreducible trees, that is, $u_1,\ldots,u_j$ are trees such that $t=u_1\cdots u_j$. Note that $u_1,\ldots,u_j$ are not necessarily irreducible. Given an order preserving map $f:P\rightarrow N(p_1(w))$ we denote by $t\cdot_{(P,f)}w$ the tree obtained by gluing the root of each tree $u_i$ at the node $f(u_i)$ of $t$ such that it is leftmost respect to the initial children of $f(u_i)$. The dual product then may be described by the following formula\[t\cdot_{\rm d}w=\sum_{(P,f)}t\cdot_{(P,f)}w.\]Note that if $t=t_1\cdots t_k$ is the unique decomposition of $t$ in irreducible trees with $|N(p_1(w))|=l+1$, then the product $t\cdot_{{\rm d}} w$ has $\binom{k+l}{k}=\binom{k+l}{l}$ addends.\[\figuretwetwo{#1}\]

\section{Hopf algebra of planar trees}\label{029}

In this section we endow the infinitesimal bialgebra of trees with a product $\tprod:\field[\T_{\infty}]\otimes\field[\T_{\infty}]\to\field[\T_{\infty}]$ such that the triple $(\field[\T_{\infty}],*,\Delta)$ is a Hopf algebra whose unit coincides with the unit of the dot product.

\subsection{Product}

Let $t,u$ be two trees with $m=|t|$ and $n=|u|$, let $P=(u_1,\ldots,u_k)$ be a $k$--partition of $u$, and let $f:P\to N(t)$ be an order preserving map. The \emph{hash product} of $t$ with $u$ respect to $(P,f)$, denoted by $t\hash{P,f}u$, is the tree obtained by gluing the root of each tree $u_i$ at the node $f(u_i)$ of $t$ such that it is rightmost respect to the initial children of $f(u_i)$. Note that $t\hash{I,f}u\in\T_{m+n}$. To short, if $P,f$ are clear in the context, we will denote $t\hash{P,f}u$ by $t\hash{}u$ instead. Note that $\treeroot\hash{}u=u$, and we set $t\hash{}\treeroot:=t$.

Since order preserving maps are injective, such a map exists only if $k\leq m+1$. We will denote by $I(P,t)$ the set of all order preserving maps from $P$ to $N(t)$. Note that $|I(P,t)|=\binom{m+1}{k}$.

\begin{exm}
Let $t$ be the tree of degree $7$ in Figure \ref{004}, let $w$ be the tree and its $3$--partition $P$ given in Figure \ref{005}, and let $f:P\to N(t)$ be the order preserving map sending the first tree of $P$ to the third node of $t$, the second one to the sixth node, and the last tree to the root. Hence, the hash product $t\hash{P,f}w$ is obtained as follows:\begin{figure}[H]
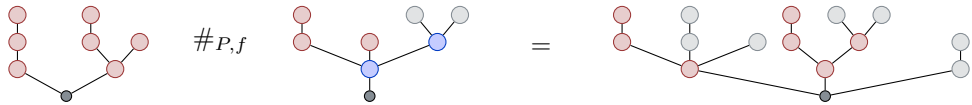
$\figuresix{#1}$\centering\caption{Hash product of trees.}\end{figure}
\end{exm}

Given two planar trees $t,w$, the \emph{product} of $t$ with $w$ is defined as follows\[t\tprod w:=\sum_{P\in\PP(w)}\,\,\sum_{f\in I(P,t)}\!\!\!t\hash{P,f}w=\sum_{P,f}t\hash{P,f}w.\]

\begin{exm}
Below the product of two trees of degrees $2$ and $3$ respectively.\begin{figure}[H]
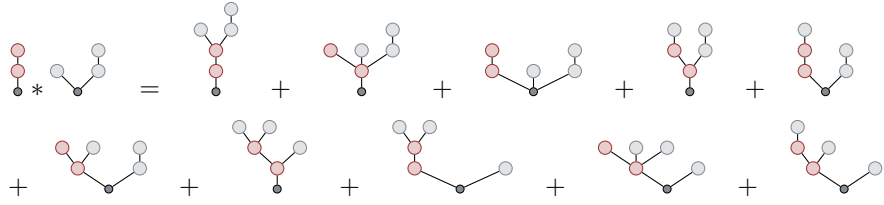
$\figuresev{#1}$\centering\caption{Product of trees.}\end{figure}
\end{exm}

To prove Proposition \ref{003} we need the following lemmas.

\begin{lem}\label{001}
For every planar trees $t,u$, the product $t\tprod u$ is formed by $\binom{m+n}{m}$ addends, where $m=|t|$ and $n=|u|$.
\end{lem}
\begin{proof}
Proposition \ref{008} implies that, for each $k\leq r:=\min(n,m+1)$, there are $\binom{n-1}{k-1}$ $k$--partitions of $u$, and for each $k$--partition there are $\binom{m+1}{k}$ order preserving maps to $N(t)$. So, there are $\binom{n-1}{k-1}\binom{m+1}{k}$ hash products of $t$ with $u$. Hence, the number of addends of $t\tprod u$ is given by the sum $\sum_{k=1}^r\binom{n-1}{k-1}\binom{m+1}{k}$. The fact that $\binom{m+n}{n}=\binom{m+n}{m}$ and the \emph{Vandermonde's identity}\footnote{Recall that the Vandermonde's identity claims that $\sum_{k=0}^r\binom{m}{k}\binom{n}{r-k}=\binom{m+n}{r}$.} imply that there are $\binom{m+n}{m}$ addends.
\end{proof}

\begin{lem}\label{002}
For planar trees $t,u,w\in\T_{\infty}$, the products $t\tprod(u\tprod w)$ and $(t\tprod u)\tprod w$ have the same number of addends.
\end{lem}
\begin{proof}
Let $m,n,r$ be the degrees of $t,u$ and $w$ respectively. Lemma \ref{001} implies that $u\tprod w$ has $\binom{n+r}{n}$ addends of degree $n+r$. Hence $t\tprod(u\tprod w)$ has $\binom{m+n+r}{m}\binom{n+r}{n}$ addends. Similarly, Lemma \ref{001} implies that $t\tprod u$ has $\binom{m+n}{m}$ addends of degree $m+n$. Hence $(t\tprod u)\tprod w$ has $\binom{m+n+r}{m+n}\binom{m+n}{m}=\binom{m+n+r}{m}\binom{n+r}{n}$ addends. Therefore $t\tprod(u\tprod w)$ and $(t\tprod u)\tprod w$ have the same number of addends.
\end{proof}

\begin{pro}\label{003}
The pair $(\field[\T_{\infty}],\tprod)$ is an algebra.
\end{pro}
\begin{proof}
To prove that $\field[\T_{\infty}]$ is an algebra we have to check the product is associative. Let $u,v,w$ be trees of degrees $m,n$ and $r$ respectively. An addend of $u\tprod(v\tprod w)$ is a tree $z:=u\hash{P,f}(v\hash{Q,g}w)$, where $Q$ is a partition of $w$, $g\in I(Q,v)$, $P$ is a partition of $v\hash{Q,g}w$ and $f\in I(P,u)$. We will show that $t$ may be written as an addend of $(u\tprod v)\tprod w$. Note that $z'=z_{N(u)\cup N(v)}$ is a subtree of $z$, and that $z'=u\hash{P_v,f'}v$, where $P_v$ is the partition of $v$ obtained by restricting $P$, that is $P_v:=\{t_{N(v)}\mid t\in P\}\backslash\{\treeroot\}$, and $f'\in I(P_v,u)$ defined by $f'(t_{N(v)})=f(t)$. Moreover, $z=z'\hash{P_w,g'}w$, where, as above, $P_w$ is the partition of $w$ obtained by restricting $P$, and $g'\in I(P_w,z')$ defined by $g'(t_{N(w)})=g(t)$ if the parent of $t$ belongs to $N^*(v)$ and $g'(t_{N(w)})=f(t)$ otherwise. Hence $u\hash{P,f}(v\hash{Q,g}w)=(u\hash{P_v,f}v)\hash{P_w,g'}w$. This fact and Lemma \ref{002} imply that $u\tprod(v\tprod w)=(u\tprod v)\tprod w$. Therefore, the product is associative.
\end{proof}

\subsection{Bialgebra}

\begin{lem}\label{025}
Let $u$ and $v$ be trees with $|u|=m$ and $|v|=n$. Then:
\begin{enumerate}
\item For a hash product $u\hash{}v$ and $k\in[m+n]_0$, there is $(i,j)\in[m]_0\times[n]_0$ such that\[(u\hash{}v)_{[1,k]}\otimes(u\hash{}v)_{[k+1,m+n]}=(u_{[1,i]}\hash{}v_{[1,j]})\otimes(u_{[i+1,m]}\hash{}v_{[j+1,n]}),\]for some hash products $u_{[1,i]}\hash{}v_{[1,j]}$ and $u_{[i+1,m]}\hash{}v_{[j+1,n]}$.
\item Reciprocally, for $(i,j)\in[m]_0\times[n]_0$ and hash products $u_{[1,i]}\hash{}v_{[1,j]}$ and $u_{[i+1,m]}\hash{}v_{[j+1,n]}$, there is a hash product $u\hash{}v$ such that\[(u_{[1,i]}\hash{}v_{[1,j]})\otimes(u_{[i+1,m]}\hash{}v_{[j+1,n]})=(u\hash{}v)_{[1,i+j]}\otimes(u\hash{}v)_{[i+j+1,m+n]}.\]
\end{enumerate}
\end{lem}
\begin{proof}
(1) Let $P$ be the partition and $f\in I(P,u)$ that define $u\hash{}v$, let $x=(u\hash{}v)_{[1,k]}$, and let $y=(u\hash{}v)_{[k+1,n+m]}$. We set $i$ and $j$ be the number of nodes of $u$ and $v$, respectively, that appear in $x$, that is, $i=|N(x_{N(u)})|$ and $j=|N(x_{N(v)})|$. Let $P_1=\{t_{N(x)}\mid t\in P\}\backslash\{\treeroot\}$ and $P_2=\{t_{N(y)}\mid t\in P\}\backslash\{\treeroot\}$. The sets $P_1$ and $P_2$ form partitions of $v_{[1,j]}$ and $v_{[j+1,n]}$ respectively. Now, we need to define $f_1\in I(P_1,u_{[1,i]})$ and $f_2\in I(P_2,u_{[i+1,m]})$. Denote by $a$ be the last node of $x$ respect to the order of $N^*(x)$. We define $f_1(t_{N(x)})=\treeroot$ if $a\in N(t_{N(x)})\cap N(v)$ and $f_1(t_{N(x)})=f(t)$ otherwise, and $f_2(t_{N(y)})=f(t)$. So, by construction, we have $x\otimes y=(u_{[1,i]}\hash{P_1,f_1}v_{[1,j]})\otimes(u_{[i+1,m]}\hash{P_2,f_2}v_{[j+1,n]})$.

(2) Let $P_1$ and $f_1$ the respective partition and map that define $u_{[1,i]}\hash{}v_{[1,j]}$, and let $P_2$ and $f_2$ the respective partition and map that define $u_{[i+1,m]}\hash{}v_{[j+1,n]}$. Denote by $w_1$ the last tree in $P_1$ and let $w_2$ be the first tree in $P_2$. To construct $u\hash{}v$ we need to define a partition $P$ of $v$ and a map $f\in I(P,u)$. For this, we distinguish three cases.

(a) If $f_1(w_1)\neq\text{root}$, we set $P=P_1\cup P_2$ and $f(t)=f_i(t)$ if $t\in P_i$ for $i\in\{1,2\}$.

(b) If $f_1(w_1)=\text{root}$ and $f_2(w_2)\neq b$, where $b$ is the first node of $u_{[i+1,m]}$, we set $P=P_1\cup P_2$ with $f(t)=f_1(t)$ if $t\in P_1\backslash\{w_1\}$, $f(w_1)=b$ and $f(t)=f_2(t)$ if $t\in P_2$.

(c) If $f_1(w_1)=\text{root}$ and $f_2(w_2)=a$, we set $P=P_1\backslash\{w_1\}\cup\{w\}\cup P_2\backslash\{w_2\}$, where $w$ is the convex tree of $v$ formed by the non trivial nodes of $w_1$ and $w_2$, with $f(t)=f_1(t)$ if $t\in P_1\backslash\{w_1\}$, $f(w)=b$ and $f(t)=f_2(t)$ if $P_2\backslash\{w_2\}$.

By construction, for each case, we obtain that required equality.
\end{proof}

\begin{thm}\label{009}
The algebra $(\field[\T_{\infty}],\tprod)$ with the coproduct $\Delta$ is a Hopf algebra.
\end{thm}
\begin{proof}
As a consequence of Lemma \ref{025} we obtain $\Delta(u\tprod v)=\Delta(u)\tprod\Delta(v)$ for all $u,v\in\T_{\infty}$. Therefore $(\field[\T_{\infty}],\tprod,\Delta)$ is a Hopf algebra.
\end{proof}

\begin{rem}\label{012}
Theorem \ref{009} and Proposition \ref{010} imply that $(\field[\T_{\infty}],*,\cdot,\Delta)$ is a $2$-associative Hopf algebra in the sense of \cite{LoRo06}.
\end{rem}

\subsection{Primitive elements}

In general, recall that for a bialgebra $H$ with coproduct $\Delta$ and counit $\epsilon$, its \emph{reduced coproduct} is the map $\bar{\Delta}:\bar{H}\to\bar{H}\otimes\bar{H}$ defined by $\bar{\Delta}(x)=\Delta(x)-1\otimes x-x\otimes1$ for all $x\in\bar{H}$, where $\bar{H}=\ker(\epsilon)$. An element $x\in H$ is \emph{primitive} if $\bar{\Delta}(x)=0$. We will denote by $\Prim(H)$ the space of all primitive elements of $H$. Since $\bar{\Delta}$ is coassociative, we can inductively define the following $(n+1)$--ary cooperation\[\bar{\Delta}^{(n)}:\bar{H}\to\bar{H}^{\otimes (n+1)},\qquad\bar{\Delta}^{(n)}=\left\{\begin{array}{ll}
\id&\text{if }n=0\\
\bar{\Delta}&\text{if }n=1\\
\bar{\Delta}^{(n)}=(\bar{\Delta}\otimes\id\otimes\cdots\otimes\id)\circ\bar{\Delta}^{(n-1)}&\text{if }n>1
\end{array}\right.\]

Set $H=\field[\T_{\infty}]$ the $2$-associative Hopf algebra described in Remark \ref{012}. Note that $\bar{H}=\bigoplus_{n\geq1}\field[\T_n]$. By applying the results obtained in \cite{LoRo06}, the primitive part of $H$, regarded as an infinitesimal bialgebra, was studied in \cite{Mq18}.

Since $H$ is a connected infinitesimal bialgebra, \cite[Proposition 2.5]{LoRo06} implies that there is an idempotent linear operator $e:\bar{H}\to\bar{H}$ defined by $e=\sum_{n\geq0}(-1)^n\cdot^n\circ\,\bar{\Delta}^{(n)}$ satisfying $e(\bar{H})=\Prim(H)$ and $e(x\cdot y)=0$ for all $x,y\in\bar{H}$. As a consequence, we obtain that the collection of $e(t)$ with $t$ an irreducible tree is a basis of $\Prim(H)$. In particular, the dimension of the $n$th component of $\Prim(H)$ is $c_{n-1}$. See \cite{Mq18} for more details.

\subsection{Labelled trees}\label{017}

Given a set $X$, an \emph{$X$--labelled tree} or simply a \emph{labelled tree} is a tree in which its non root nodes are labelled by elements of $X$. Note that the number of $X$--labelled trees of degree $n$ is $|X|^nc_n$ whenever $X$ is a finite set. We will denote by $\T_n(X)$ the set of all $X$--labelled trees of degree $n$, and we set\[\T(X)=\bigsqcup_{n\geq0}\T_n(X).\]

The structure of trees on $K[\T_{\infty}]$ induces a $2$-associative Hopf algebra structure on the $K$--vector space spanned by the set of all labelled trees $\T(X)$, that is, the product and coproduct are obtained by operating the underlying structures of the trees and keeping the labels of the nodes. Note that $\T_{\infty}$ may be regarded as the set of $\emptyset$--labelled trees, so we will keep the notation $*$ and $\Delta$ for the product and coproduct of labelled trees. For instance, below the product of two labelled trees of degrees $2$ and $3$ respectively.\begin{figure}[H]
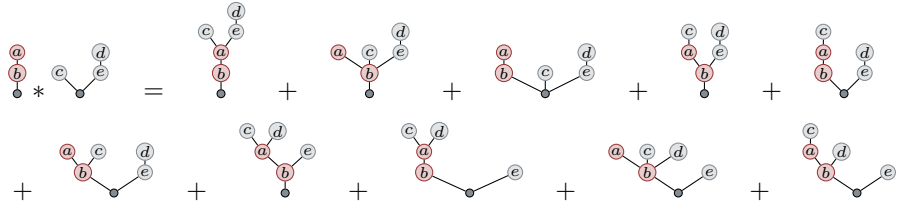
$\figureeig{#1}$\centering\caption{Product of labelled trees.}\end{figure}Similarly, the coproduct of a labelled tree of degree $6$.\begin{figure}[H]
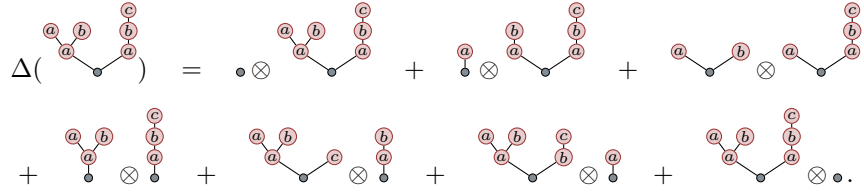
$\figurenin{#1}$\centering\caption{Coproduct of a labelled tree.}\end{figure}

In the same way as in $\field[\T_{\infty}]$, the primitive part of $\field[\T(X)]$ is spanned by the collection of $e(t)$ with $t$ an irreducible labelled tree. In particular, the dimension of the $n$th component of $\Prim(\field[\T(X)])$ is $|X|^nc_{n-1}$ whenever $X$ is finite.

\section{Labelled $n$--trees}\label{031}

An \emph{$n$--tree} is a $[n]$--labelled tree of degree $n$ whose labels admit no repetition. For every nonnegative $n$, the number of $n$--trees is $n!c_n$. We will denote by $\T[n]$ the set of all $n$--trees, and we set\[\T[\infty]=\bigsqcup_{n\geq0}\T[n].\]

Given a $\N$--labelled tree $t$ of degree $n$, we denote by $\s(t)$ the \emph{standardization} of $t$ defined by an order preserving relabelling of its nodes such that it is a $n$--tree. For instance, below the standardization of an $\N$--labelled tree of degree $5$:\begin{figure}[H]
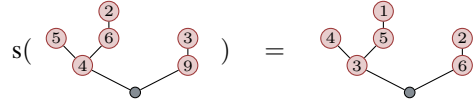
$\figuretwe{#1}$\centering\caption{Standardization of a tree.}\end{figure}

The coproduct of an $n$--tree $t$, denoted by $\Delta_{\s}(t)$, is defined as the sum obtained by standardizing the terms of each tensor in $\Delta(t)$, that is\[\Delta_{\s}(t):=\sum_{k=0}^n\s(t_{[1,k]})\otimes\s(t_{[k+1,n]})\]By using Sweedler's notation, we have $\Delta_{\s}(t)=\s(t_{(1)})\otimes\s(t_{(2)})$. Note that $\Delta_{\s}=(\s\otimes\s)\circ\Delta$. For instance, below the coproduct of a $6$--tree.\begin{figure}[H]
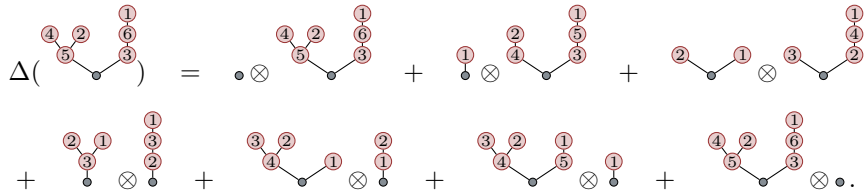
$\figuretenthr{#1}$\centering\caption{Coproduct of a $6$--tree.}\end{figure}

\begin{pro}\label{013}
The pair $(\field [\T[\infty]], \Delta_{\s})$ is a coalgebra.
\end{pro}
\begin{proof}
 Since $\Delta$ is a coproduct and $\s(\s(w)_A)=\s(w_A)$ for all $\N$--labelled tree $w$ and for all $A\subseteq N(w)$, then, for every $m$--tree $t$, we have\[\begin{array}{rcl}
(\id\otimes\Delta_{\s})(\Delta_{\s}(t))&=&(\id\otimes\Delta_{\s})(\s(t_{(1)})\otimes\s(t_{(2)})),\\
&=&\s(t_{(1)})\otimes\s(\s(t_{(2)})_{(1)})\otimes\s(\s(t_{(2)})_{(2)}),\\
&=&\s(t_{(1)})\otimes\s((t_{(2)})_{(1)})\otimes\s((t_{(2)})_{(2)}),\\
&=&(\s\otimes\s\otimes\s)(\id\otimes\Delta)(\Delta(t)),\\
&=&(\s\otimes\s\otimes\s)(\Delta\otimes\id)(\Delta(t)),\\
&=&\s((t_{(1)})_{(1)})\otimes\s((t_{(1)})_{(2)})\otimes\s(t_{(2)}),\\
&=&\s(\s(t_{(1)})_{(1)})\otimes\s(\s(t_{(1)})_{(2)})\otimes\s(t_{(2)}),\\
&=&(\Delta_{\s}\otimes\id)(\s(t_{(1)})\otimes\s(t_{(2)})),\\
&=&(\Delta_{\s}\otimes\id)(\Delta_{\s}(t)).
\end{array}\]Therefore $\field [\T[\infty]]$ is a coalgebra.
\end{proof}

For a positive integer $m$, the \emph{$m$--shift} or simply the \emph{shift} of an $\N$--labelled tree $t$, denoted by $t[m]$, is the $\N$--labelled tree obtained by adding $m$ to the labels of the non root nodes of $t$. For instance, below the $4$--shift of an $\N$--labelled tree.\begin{figure}[H]
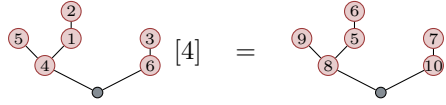
$\figureten{#1}$\centering\caption{$4$--shift of a tree.}\end{figure}Note that $\s(t)=\s(t[k])$ for all $\N$--labelled tree $t$ and for all $k\in\N$. Furthermore if $|t|=m$ and $w$ is another $\N$--labelled tree such that the labels of $t$ are smaller than the labels of $w$, then\begin{equation}\label{015}\s(t\cdot w)=\s(t)\cdot \s(w)[m]\quad\text{and}\quad\s(t*w)=\s(t)*\s(w)[m].\end{equation}

We may adapt the dot product, via the shift operation, in order to obtain an infinitesimal structure on the bialgebra $K[\T[\infty]]$ and its respective subspaces. More specifically, for trees $t,w\in\T[\infty]$ with $|t|=m$, we consider the product $t/w:=t\cdot w[m]$ for all $t,w\in\T[\infty]$. For instance, we have\begin{figure}[H]
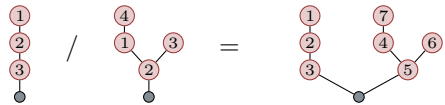
$\figuretensix{#1}$\centering\caption{Standardized dot product of trees.}\end{figure}We say that $t\in\T[\infty]$ is \emph{reducible}, respect to the product $/$, if there are non trivial trees $u,v\in\T[\infty]$ such that $t=u/v$, otherwise $t$ is called irreducible. Note that irreducible trees of $\field[\T[\infty]]$ are not necessarily irreducible in $\field[\T(\N)]$. For instance, the tree below is an irreducible $3$--tree but no an $\N$-labelled irreducible tree.\begin{figure}[H]
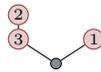
$\figuretensev{#1}$\centering\caption{Irreducible $3$--tree.}\end{figure}Furthermore, every non trivial $n$-tree $t$ can be uniquely written as $t=t_1/\cdots/t_k$, where $t_1,\ldots,t_k$ are irreducible.

\begin{pro}\label{011}
The tuple $(\field[\T[\infty]],/,\Delta_{\s})$ is a connected infinitesimal bialgebra. In consequence, $\field[\T[\infty]]$ is cofree.
\end{pro}
\begin{proof}
We must verify the infinitesimal condition of $/$ with $\Delta_{\s}$. Let $t,w\in \T[\infty]$ with $|t|=m$. Indeed, Equation \ref{015} and the fact that $\Delta(w[m])=w_{(1)}[m]\otimes w_{(2)}[m]$ for all $\N$--labelled tree $t$ and for all $k\in\N$, imply that\[\begin{array}{rcl}
\Delta_{\s}(t/w)&=&(\s\otimes\s)(\Delta(t\cdot w[m])),\\
&=&(\s\otimes\s)(\Delta(t)\cdot w[m]+t\cdot\Delta(w[m])-t\otimes w[m]),\\
&=&(\s\otimes\s)(t_{(1)}\otimes t_{(2)}\cdot w[m]+t\cdot w_{(1)}[m]\otimes w_{(2)}[m]-t\otimes w[m]),\\
&=&\s(t_{(1)})\otimes \s(t_{(2)})\cdot w[|t_{(2)}|]+t\cdot \s(w_{(1)})[m]\otimes \s(w_{(2)})-t\otimes w,\\
&=&\s(t_{(1)})\otimes \s(t_{(2)})/ w+t/\s(w_{(1)})\otimes \s(w_{(2)})-t\otimes w,\\
&=&\Delta_{\s}(t)/ w+t/\Delta_{\s}(w)-t\otimes w.\\
\end{array}\]
Furthermore, Proposition \ref{026} implies that $\field[\T[\infty]]$ and $T(\Prim(\field[\T[\infty]]))$ are isomorphic as infinitesimal bialgebras, therefore $\field[\T[\infty]]$ is cofree. 
\end{proof}

In what follows we gives $\field[\T[\infty]]$ with a structure of $2$-associative Hopf algebra.

The \emph{star product} of two trees $t,w\in\T[\infty]$, denoted by $t\star w$, is defined by multiplying $t$ with the $m$--shift of $w$ as $\N$--labelled trees, where $m$ is the degree of $t$, that is, $t\star w=t*w[m]$. By extending linearly the $m$--shift operation, we obtain $(t*w)[m]=t[m]*w[m]$ for all $\N$--labelled trees $t,w$. This and the fact that $*$ is associative imply that $\star$ is associative as well, indeed if $t,u,w\in\T[\infty]$ with $|t|=m$ and $|u|=n$, then \begin{equation}\label{014}(t\star u)\star w=(t*u[m])*w[m+n]=t*(u[m]*w[n][m])=t*(u*w[n])[m]=t*(u\star w)[m]=t\star(u\star w).\end{equation}Below the star product of a $2$--tree with a $3$--tree.\begin{figure}[H]
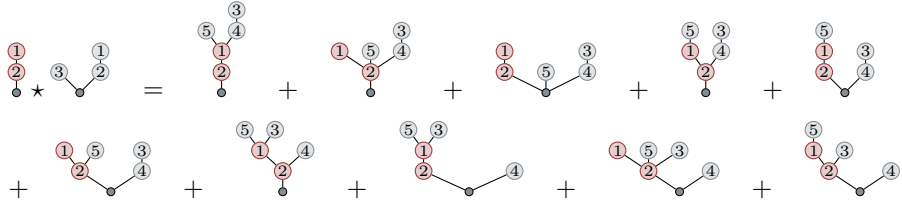
$\figureele{#1}$\centering\caption{Star product of two trees.}\end{figure}

\begin{thm}\label{020}
The tuple $(\field[\T[\infty]],\star,\Delta_{\s})$ is a graded Hopf algebra.
\end{thm}
\begin{proof}
The pair $(\field[\T[\infty]],\star)$ is an algebra because of Equation \ref{014}. Because of Proposition \ref{013}, it is enough to verify that $\Delta_s$ is an algebra homomorphism. Indeed, since $\Delta(w[m])=w_{(1)}[m]\otimes w_{(2)}[m]$ holds for all $\N$--labelled tree $t$ and for all $m\in\N$, Equation \ref{015} implies the following\[\begin{array}{rcl}
\Delta_{\s}(t\star w)&=&\Delta_{\s}(t*w[m]),\\
&=&(\s\otimes\s)(\Delta(t)*\Delta(w[m])),\\
&=&(\s\otimes\s)((t_{(1)}\otimes t_{(2)})*(w_{(1)}[m]\otimes w_{(2)}[m])),\\
&=&(\s\otimes\s)((t_{(1)}*w_{(1)}[m])\otimes(t_{(2)}*w_{(2)}[m])),\\
&=&(\s(t_{(1)})*\s(w_{(1)})[m_1])\otimes(\s(t_{(2)})*\s(w_{(2)})[m_2]),\\
&=&(\s(t_{(1)})\star\s(w_{(1)}))\otimes(\s(t_{(2)})\star\s(w_{(2)})),\\
&=&(\s(t_{(1)})\otimes\s(t_{(2)}))\star(\s(w_{(1)})\otimes\s(w_{(2)})),\\
&=&\Delta_{\s}(t)\star\Delta_{\s}(w).
\end{array}\]Therefore $(\field[\T[\infty]],\star,\Delta_{\s})$ is a graded Hopf algebra.
\end{proof}

Proposition \ref{011} and Theorem \ref{020} imply that $(\field[\T[\infty]],\star,/,\Delta_{\s})$ is a $2$-associative Hopf algebra.

\subsection{Increasing trees}

An \emph{increasing tree} of degree $n$ is an $n$--tree in which every path of labels starting at the root is increasing. We will denote by $I[n]$ the set of all increasing trees of degree $n$, and we set\[I[\infty]:=\bigsqcup_{n\geq0}I[n].\]In the literature, the trees above are known as \emph{planar increasing trees}. 

The set of increasing trees of degree $n$ can be constructed inductively by gluing a node labelled by $n$ on the nodes of the increasing trees of degrees $n-1$. This implies that the number of increasing trees is the double factorial $(2n-1)!!=1\cdot3\cdots(2n-1)$ (see \cite{Ja08}).

\begin{pro}\label{016}
The vector space $\field[I[\infty]]$ is a graded Hopf subalgebra of $\field[\T[\infty]]$.
\end{pro}
\begin{proof}
This is a consequence of the fact that the shifts, the convex trees and the standardizations of planar increasing trees are trees with increasing paths, then $\field[I[\infty]]$ is closed under the star product and the coproduct of $\field[T[\infty]]$.
\end{proof}

\subsection{Sorted trees}\label{022}

A \emph{sorted tree} is an increasing tree in which the labels of the children of each node are ordered from left to right. These trees can be regarded as a planar representation of the non planar increasing trees also known as Cayley increasing trees or recursive trees (see \cite{St97}). We will denote by $SI[n]$ the set of all sorted trees of degree $n$, and we set\[SI[\infty]:=\bigsqcup_{n\geq0}SI[n].\]

The set of sorted trees of degree $n$ can be constructed inductively by gluing a node labelled by $n$ on the nodes of the sorted trees of degrees $n-1$ so that it is the rightmost node respect to its siblings. This implies that the number of sorted trees is $n!$ (see \cite{St97}).

Below the set of sorted trees of degree $3$:\begin{figure}[H]
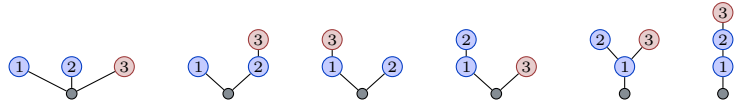
$\figuretenfou{#1}$\centering\caption{Sorted trees of degree $3$.}\end{figure}

By proceeding analogously as in Proposition \ref{016}, we obtain the following result.

\begin{pro}\label{024}
The vector space $\field[SI[\infty]]$ is a graded sub $2$-associative Hopf algebra of $\field[\T[\infty]]$.
\end{pro}

Another Hopf algebra was constructed in \cite{GrLa89} by using a similar family of trees called heap ordered trees.

\subsection{Primitive elements}

Theorem \ref{026} implies that the 2-associative Hopf algebra $\field[\T[\infty]]$ and their sub algebras are isomorphic, as infinitesimal bialgebras, to the tensor algebra on the vector space spanned by their primitive elements. In particular, they are cofree among connected coalgebras. The primitive part is determined by the action of the linear operator $e=\sum_{n\geq0}(-1)^n/^n\circ\,\bar{\Delta}^{(n)}$ on their irreducible trees respect to the product $/$. For instance, the following trees form a basis for the primitive sorted trees of degree $3$:\begin{figure}[H]$\figureteneig{#1}$\centering\caption{Basis for primitive sorted trees of degree $3$.}\end{figure}

We can determine the dimensions of the components of their primitive parts, or equivalently, the number of irreducible trees in each degree, as follows: For a connected graded infinitesimal bialgebra $H=\bigoplus_{n\geq0}H_n$ with $\Prim(H)=\bigoplus_{n\geq0}P_n$ we have the \emph{generating functions} $f(H,x)=\sum_{n\geq0}a_nx^n$ and $g(H,x)=\sum_{n\geq0}b_nx^n$ where $a_n=\dim(H_n)$ and $b_n=\dim(P_n)$ for all $n\geq0$. By \cite[Section 6.3]{Le72}, the series $f$ and $g$ are related through the equation $f(H,x)=(1-g(H,x))^{-1}$. Since $a_0=1$ and $b_0=0$ we obtain the following recursive formula\begin{equation}\label{018}b_n=a_n-\sum_{k=1}^{n-1}a_kb_{n-k}.\end{equation}

The following is a consequence of Equation \ref{018}.

\begin{pro}
We have:
\begin{enumerate}
\item[(1)]$g(\field[T[\infty]],x)=x+3x^2+23x^3+271x^4+4251x^5+82967x^6+\cdots$
\item[(2)]$g(\field[I[\infty]],x)=x+2x^2+10x^3+74x^4+706x^5+8162x^5+110410x^6+\cdots$
\item[(3)]$g(\field[SI[\infty]],x)=x+x^2+3x^3+13x^4+71x^5+461x^6+3447x^7+\cdots$
\end{enumerate}
\end{pro}

\section{Planar trees and permutations}\label{032}

For a finite set $X$ we will denote by $X\sqcup X$ the multiset in which each element of $X$ appears twice.

In \cite{Ja08}, Janson showed an explicit bijection between the set of increasing trees and the so called Stirling permutations which are obtained through the depth first walk traversing of a tree. In this section, we use this technique to code $n$--trees as permutations of the multiset $[n]\sqcup[n]$.

In this work, a permutation of a finite multiset $X$ is a word in which each letter of $X$ appears as times as its multiplicity. For a finite set $X$, a \emph{$2$--permutation} of $X$ is a permutation of the multiset $X\sqcup X$. We will denote by $\Sym_{n,2}$ the set of all $2$--permutations of $[n]$. For every word $u$ formed by positive integers and every nonnegative integer $m$, we will denote by $u[m]$ the \emph{$m$--shift} of $u$ obtained by adding $m$ to each letter of it. For instance, the $3$--shift of the $2$--permutation ${\tt213132}$ is ${\tt435354}$. 

Given a word $u$ formed by letters of some subset of $\N$, we denote by $\s(u)$ the \emph{standardization} of $u$ defined by replacing the occurrences of the $k$th smallest letter of $u$ by the number $k$. For instance $\s({\tt299552663883})={\tt166331442552}$.

\subsection{Treed permutations}

A \emph{treed permutation} of size $n$ is a $2$--permutation of $[n]$ in which the elements placed between the two occurrences of a letter form a $2$-permutation as well, that is, each one of these elements appear twice. For instance, the word $u={\tt211442665335}$ is a treed permutation of size $4$. We will denote by $\Sym_{n,2}^T$ the set of all treed permutations of size $n$, and set\[\Sym_{\infty,2}^T=\bigsqcup_{n\geq0}\Sym_{n,2}^T.\]Note that every treed permutation can be written as a concatenation of words of the form $kuk$, called \emph{parts}, where $k\in[n]$ and $\s(u)$ is a treed permutation. For instance, $u$ above is the concatenation of the words ${\tt211442}$, ${\tt66}$ and ${\tt5335}$.

The \emph{depth first walk} of an $n$--tree $t$, denoted by $\varepsilon(t)$, is the $2$--permutation obtained by clockwise traversing the tree from root to root across its border, so that each label is read twice. This traversing method is also known as the \emph{Euler tour} of the tree and the word obtained is called its \emph{Euler tour representation}. For instance, the treed permutation ${\tt133156446225}$ is obtained by traversing the following $6$--tree:\begin{figure}[H]
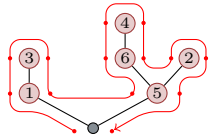
$\figuretennin{#1}$\centering\caption{Euler tour of a tree.}\end{figure}Note that if $t$ is an $\N$--labelled tree such that $t=t_1\cdots t_k$ for some $t_1,\ldots,t_k\in T(\N)$, then $\varepsilon(t)=\varepsilon(t_1)\cdots\varepsilon(t_k)$.

For an $\N$-labelled tree $t$ and an integer $k$, we will denote by $t\circ k$ the tree obtained by identifying the root of $t$ with $\ktreeone$.

\begin{pro}\label{021}
For every positive integer $n$, $\T[n]$ and $\Sym_{n,2}^T$ are bijective via the Euler tour representation. In consequence, the number of treed permutations of size $n$ is $n!c_n$.
\end{pro}
\begin{proof}
By definition, it is clear that, for every $n$-tree $t$, the word $\varepsilon(t)$ is a uniquely defined treed permutation of size $n$. Let $u$ be a treed permutation of size $n$. We will show, by induction on $n$, that there is an $n$--tree $t$ such that $\varepsilon(t)=u$. If $n=1$, then $u={\tt11}$ which is obtained as the Euler tour representation of the unique $1$--tree. For $n>1$, we assume every treed permutation of size smaller than $n$ can be obtained as the Euler representation of a tree in $\T[\infty]$. Let $k_1u_1k_1,\ldots,k_ru_rk_r$ be the parts of $u$. If $r=1$, inductive hypothesis implies that there is an $\N$-labelled tree $t_1$ such that $\varepsilon(t_0)=u_1$. Hence $\varepsilon(t)=u$ where $t=t_0\circ k_1$. If $r>1$, inductive hypothesis implies that there are irreducible $\N$-labelled trees $t_1,\ldots,t_r$ such that $\varepsilon(t_i)=k_iu_ik_i$ for all $i\in[r]$. Hence $\varepsilon(t)=u$ where $t=t_1\cdots t_r$. Therefore, the restriction of $\varepsilon$ from $\T[n]$ to $\Sym_{n,2}^T$ is onto. We will show now, by induction on $n$, that this restriction is injective. If $n=1$ it is obvious. For $n>1$, we assume the claim is true for values smaller than $n$. Let $t,w$ be two $n$--trees with $t_1\cdots t_p$ and $w_1\cdots w_q$ their irreducible decompositions, and such that $\varepsilon(t)=\varepsilon(w)$. Note that the roots of $t$ and $w$ have the same number of children because $\varepsilon(t)$ and $\varepsilon(w)$ must have the same number of parts, so $p=q$. If $p=1$, then $\varepsilon(t)=kuk=\varepsilon(w)$ for some $k\in\N$. Hence, there are $\N$-labelled trees $t',w'$ satisfying $t=t'\circ k$ and $w=w'\circ k$ such that $\varepsilon(\s(t'))=\varepsilon(\s(w'))$. By inductive hypothesis, we have $t'=w'$ which implies that $t=w$. If $p>1$, inductive hypothesis implies that $t_i=w_i$ because $\varepsilon(\s(t_i))=\varepsilon(\s(w_i))$ for all $i\in[p]$. Therefore $t=w$.
\end{proof}

Proposition above and Theorem \ref{020} imply that $\field[\T[\infty]]$ can be released as a Hopf algebra of treed permutations, so that that $\varepsilon:\T[\infty]\to\Sym_{n,2}^T$ extends to an isomorphism from $\field[\T[\infty]]$ onto $\field[\Sym_{n,2}^T]$. In what follows of this subsection we will characterize the product and coproduct of $\field[\Sym_{n,2}^T]$ through this bijection.

For a word $u$ and a set $A$, we denote by $u_A$ the word obtained by removing the letters that not belong to $A$.

Given a subset $X$ of non negative integers and a $2$--permutation $u$ of it, we will denote by $X^u$ the set $X$ totally ordered respect to the positions of the second occurrences of the letters of $X$ in $u$. A $k$--\emph{partition} of a $2$--permutation $u$ is an ordered collection $P=(u_1,\ldots,u_k)$ of words called \emph{blocks} of $P$, where $u_i=u_{I_i}$ for some $k$--partition $(I_1,\ldots,I_k)$ of $X^u$. The set of all $k$--partitions of $u$ is denoted by $\PP_k(u)$, and set\[\PP(u)=\PP_1(u)\sqcup\cdots\sqcup\PP_n(u).\]For instance, the collection $P=({\tt 11},{\tt2552},{\tt4334})$ is a $3$--partition of $u={\tt 21{\red1}5{\red52}43{\red34}}$. We will denote by $X^u_0$ the ordered set obtained by adding $0$ as the maximal element of $X^u$.

Let $u,w$ be two treed permutations of sizes $m$ and $n$ respectively, let $P=\{w_1,\ldots,w_k\}$ be a $k$--partition of $w[m]$, and let $f:P\to[m]^u_0$ be an order preserving map. The \emph{hash product} of $u$ with $w$, denoted by $u\hash{P,f}w$, is the treed permutation obtained by inserting each block $w_i$ in $u$ as follows: if $f(w_i)\neq0$ we put $w_i$ just before the second occurrence of $f(w_i)$, otherwise we put $w_k$ at the end of $u$. Note that $u\hash{P,f}w$ is a treed permutation of size $m+n$. To short, if $P,f$ are clear in the context, we will denote $u\hash{P,f}w$ by $u\hash{}w$ instead.

Since order preserving maps are injective, such a map exists only if $k\leq m+1$. We will denote by $I(P,u)$ the set of all order preserving maps from $P$ to $[m]_0^u$.

\begin{exm}
Let $u={\tt42{\red2}31{\red134}}$ and $w={\tt1{\blue1}23{\blue3}54{\blue452}}$ be treed permutations of sizes $4$ and $5$. Let $P=\{w_1,w_2,w_3\}$ be a $3$--partition of $w[4]={\tt5{\blue5}67{\blue7}98{\blue896}}$, where $w_1={\tt5577}$, $w_2={\tt9889}$ and $w_3={\tt66}$, and let $f:P\to[4]_0^u$ defined by $f(w_1)=2$, $f(w_2)=3$ and $f(w_3)=0$. Hence $u\hash{}w={\tt42{\red5577}2311{\red9889}34{\red66}}$.
\end{exm}

Given two treed permutations $u,w$, the \emph{product} of $u$ with $w$ is defined as follows\[u*w:=\sum_{P\in P(w[m])}\sum_{f\in I(P,u)}u\hash{P,f}w=\sum_{P,f}u\hash{P,f}w.\]

\begin{exm}
Below the product of two treed permutations.\[\begin{array}{rcl}
{\tt2112}*{\tt\red332112}&=&{\tt2112}{\tt\red554334}+{\tt211{\tt\red554334}2}+{\tt21{\tt\red554334}12}+{\tt21{\tt\red5533}12{\tt\red44}}+{\tt21{\tt\red5533}1{\tt\red44}2}\\[1.5mm]&+&{\tt211{\tt\red5533}2{\tt\red44}}+{\tt21{\tt\red55}12{\tt\red4334}}+{\tt21{\tt\red55}1{\tt\red4334}2}+{\tt211{\tt\red55}2{\tt\red4334}}+{\tt21{\tt\red55}1{\tt\red33}2{\tt\red44}}.\end{array}\]
\end{exm}

The coalgebra structure of $\field[\Sym_{\infty,2}^T]$ is given by the following coproduct\[\Delta(u)=\sum_{k=0}^n\s(u_{[x_1,x_i]})\otimes\s(u_{[x_{i+1},x_n]})\]for all $u\in\Sym_{n,2}^T$ with $X^u=\{x_1<\cdots<x_n\}$.

\begin{exm}
Below the coproduct of a treed permutation of size $5$.\[\begin{array}{rcl}
\Delta({\tt215{\red5}43{\red3412}})&=&\emptyset\otimes{\tt2155433412}+{\tt11}\otimes{\tt21433412}+{\tt2211}\otimes{\tt213312}\\[1.5mm]
&+&{\tt332211}\otimes{\tt2112}+{\tt14432231}\otimes{\tt11}+{\tt2155433412}\otimes\emptyset.
\end{array}\]
\end{exm}

\subsection{Stirling permutations}

An \emph{Stirling permutation} of size $n$ is a $2$--permutation of $[n]$ in which $h>k$ for all element $h$ placed between the two occurrences of a letter $k$, see \cite{GeSt78} for details. Note that every Stirling permutation is a treed permutation as well. For instance, the word ${\tt13344122}$ is an Stirling permutation of size $4$. We will denote by $\Sym_{n,2}^S$ the set of all Stirling permutations of size $n$, and set\[\Sym_{\infty,2}^S=\bigsqcup_{n\geq0}\Sym_{n,2}.\]Every Stirling permutation of size $n$ can be obtained by intercalating ${\tt nn}$ between the letters of a Stirling permutation of size $n-1$. For instance, the Stirling permutations of size $2$ are obtained by inserting ${\red22}$ in ${\tt11}$, that is ${\tt11{\red22}}$, ${\tt1{\red22}1}$ and ${\tt{\red22}11}$. Hence $|\Sym_{n,2}^S|=(2n-1)!!$

As was showed by Janson \cite{Ja08}, the restriction of $\varepsilon$ defines a bijection between $I[n]$ and $\Sym_{n,2}^S$ for all positive integer $n$. This bijection and Proposition \ref{016} imply that $\field[I[\infty]]$ can be released as a Hopf subalgebra of $\field[\T[\infty]]$ spanned by Stirling permutations, so that the restriction of $\varepsilon$ is an isomorphism from $\field[I[\infty]]$ onto $\field[\Sym_{\infty,2}^S]$. 

Note that the standardizations of the blocks of the $k$--partitions of a Stirling permutation are Stirling permutations as well. This and the fact that the shift operations appear in the definition of the hash product of treed permutations, implies that the hash product of two Stirling permutations is again an Stirling permutation. Hence, the product and coproduct in $\field[\Sym_{\infty,2}^S]$ is described in the same way as in $\field[\Sym_{\infty,2}^T]$.

\subsection{Classic permutations}

In this subsection we give a bijection between sorted trees and permutations.  Through this map we may consider of Hopf algebra of sorted trees described in Subsection \ref{022} as a realisation of the Malvenuto--Reutenauer Hopf algebra of permutations \cite{MaRe95,AgSo02} via trees. In \cite{GrLa09}, it is used another bijection to represent the Hopf algebra of heap ordered trees in terms of permutations.

The traverse of the labels of an in increasing tree, via the depth-first post order, induces a unique permutation of size the degree of the tree. For instance, the tree below induces the permutation ${\tt2413}$:\begin{figure}[H]
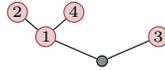
$\figuretwenty{#1}$\centering\caption{Sorted tree inducing ${\tt2413}$.}\end{figure}

\begin{pro}\label{023}
The depth-first post traverse defines an explicit bijection between $SI[n]$ and $\Sym_n$ for all integer $n>0$.
\end{pro}
\begin{proof}
Since $|SI[n]|=n!=|\Sym_n|$ it is enough to show that $\phi:S_n\to\Sym_n$, defined via the depth-first post order, is surjective, that is, for a permutation $u=u_1\cdots u_n$ of $[n]$ there is a sorted tree $t$ such that $\phi(t)=u$. We will proceed by induction on $n$. If $n=1$ the claim is obvious. If $n>1$, we assume the result is true for $n-1$. Inductive hypothesis implies that $\phi(w)=u_{[n-1]}=u_1\cdots u_{k-1}u_{k+1}\cdots u_n$ for some sorted tree $w\in SI[n-1]$, where $k$ is the unique index satisfying $u_k=n$. We define $t$ as the sorted tree of degree $n$ obtained by adding an $n$ labelled rightmost child either to the root of $w$ if $k=n$ or to the node of $w$ labelled by $u_{k+1}$ otherwise. Hence $\phi(t)=u$. Therefore $\phi$ is bijective.
\end{proof}

The result above defines an explicit bijection between the collections of sorted trees and classic permutations. A similar bijection was given in \cite{AgSo05} via the depth-first pre order traverse of trees.

Proposition \ref{023} and Proposition \ref{024} imply that $\field[SI[\infty]]$ can be released as a Hopf algebra of permutations. This structure coincides with the Hopf algebra of permutations introduced in \cite{MaRe95}. Indeed, for sorted trees $t,w$ with respective permutations $u$ and $v$, the traverse of a hash product $t\hash{P,f}w[m]$ where $m=|t|$ defines a permutation obtained by shuffling the letters of $u$ with a partition of $v[m]$ induced through of $P$ and $f$. Since there are $\binom{m+n}{m}$ hash products as above, then the product of $u$ with $v$ is exactly the product defined in \cite{MaRe95}. The equivalence of the coproduct is immediate from the bijection.

\begin{rem}[Signed permutations]
Other versions of the Malvenuto--Reutenauer Hopf algebra can be released via sorted trees. The Hopf algebra of signed permutations, constructed in \cite{FoFr17}, and its sub Hopf algebra of even--signed permutations can be obtained by considering sorted trees whose nodes are additionally labelled by two possible colors. For instance, the signed permutation $[-3,1,4,-5,2]$ can be represented by the following sorted tree:\begin{figure}[H]
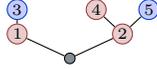
$\figuretweone{#1}$\centering\caption{Sorted double labelled tree.}\end{figure}
\end{rem}

\section{Binary trees}\label{033}

In this section, we show that $\field[\T_{\infty}]$, with Hopf algebra structure given in Section \ref{029}, can be regarded as a realization of the dual structure of the Loday--Ronco Hopf algebra of binary trees \cite{LoRo98}, which was studied in \cite{AgSo04}.

Recall that a \emph{binary tree} is a connected planar graph, with no cycles, in which each internal vertex is trivalent. We will denote by $\deg(x)$ the number of internal vertices of a binary tree. The set of binary trees with $n$ internal vertices is denoted by $Y_n$, and we set\[Y_{\infty}:=\bigsqcup_{n\geq0}Y_n.\]For instance:\begin{figure}[H]$\figuretwethr$\centering\caption{Binary trees with at most $2$ internal vertices.}\end{figure}

It was showed in \cite{LoRo98} that the vector space $\field[Y_{\infty}]$ can be given with a structure of graded Hopf algebra that is regarded as a Hopf subalgebra of $\field[\Sym_{\infty}]$.

Given a binary tree $x\in Y_n$, we numerate its $(n+1)$ leaves from $0$ to $n$. Note that given a leaf of $x$ there is a unique branch associated to this leaf. We denote by $x_{[i,j]}$ the subtree of $x$ placed between the branches $i$ and $j$. For instance:
\begin{figure}[H]
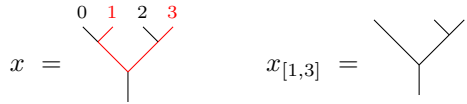
$\figuretwefou$\centering\caption{Convex binary tree.}\end{figure}

For binary trees $x,y$, the \emph{over product} of $x$ with $y$, denoted by $x/y$, is the binary tree obtained by identifying the root $x$ with the leftmost leaf of $y$. Similarly, the \emph{under product} of $x$ with $y$, denoted by $x\backslash y$, is obtained by identifying the root of $y$ with the rightmost leaf of $x$. A binary tree $x$ is \emph{irreducible} respect to the under product if $x=y/\binarytreeone$ for some binary tree $y$. Note that every binary tree $x$ can be uniquely written as $x=x_1\backslash \cdots\backslash x_k$, where $x_1,\ldots x_k$ are $\backslash$-irreducible trees, with $k\geq1$.

In \cite{AgSo04}, the decomposition above is used to define a bijection $\varphi:\T_{\infty}\rightarrow Y_{\infty}$ in order to establish an explicit isomorphism between the non commutative Connes-Kreimer Hopf algebra of Foissy and the Loday-Ronco Hopf algebra of binary trees.

To describe the map $\varphi$, we recall that if $t\in\T_{\infty}$ is irreducible of degree greater than one, then it has the form $t=t'\circ\treeone{}$ for some $t'\in\T_{\infty}$. The map $\varphi$ is defined inductively as follows\[\varphi(t)=\left\{\begin{array}{ll}
|&\text{if }t=\treeroot,\\[1mm]
\binarytreeone&\text{if }t=\treeone{},\\[1mm]
\varphi(t')/\binarytreeone&\text{if }t\text{ is irreducible},\\[1mm]
\varphi(t_1)\backslash\cdots \backslash \varphi(t_k)&\text{if }t=t_1\cdots t_k.
\end{array}\right.\]For instance:\begin{figure}[H]
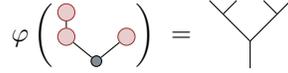
$\figuretwefiv{#1}$\centering\caption{Binary tree of a tree.}\end{figure}

Consider now the Hopf algebra $(\field[\T_{\infty}],*,\Delta)$ introduced in the Section \ref{029}. We will show that the Hopf algebra structure induced on $\field[Y_{\infty}]$ through the map $\varphi$ coincides with the dual graded Hopf algebra of Loday--Ronco described as in \cite{AgSo04}.

If $t$ is a non trivial tree in $\T_{\infty}$, then the map $\varphi$ leads a correspondence between $N^*(t)$ and the internal vertices of the binary tree $\varphi(t)$. In this way, the order considered on $N^*(t)$ induces an order on the internal vertices of $\varphi(t)$. For instance:\begin{figure}[H]
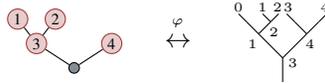
$\figuretwesix{#1}$\centering\caption{Traversing correspondence.}\end{figure}

This order on the internal vertices of a binary tree is known as \emph{depth first in order}, that is, we first traverse the left subtree of a node, the node itself and so the right subtree. We use such order to describe the Hopf algebra structure on $\field[Y_{\infty}]$. Let $t\in\T_{\infty}$ and let $x\in Y_{\infty}$ such that $\varphi(t)=x$. Recall that a convex tree $t_{[i,j]}$ of $t$ is determined by an interval of $N^*(t)$. The image of $t_{[i,j]}$ under $\varphi$ is given by the subtree of $x$ with internal vertices between $i$ and $j$, that $\varphi(t_{[i,j]})=x_{[i-1,j]}$. Hence, we can obtain partitions of $x$ via the partitions of $t$, more explicitly, a \emph{$k$-partition} of $x$ will be a collection of binary trees $x_{[0,n_1]},x_{[n_1,n_2]},\ldots,x_{[n_k,n]}$, where $0<n_1<\cdots<n_k<n$.

Under the description given above, we describe the Hopf algebra structure on $\field[Y_{\infty}]$.

For a binary tree $x\in Y_n$, its coproduct is\[\Delta(x)=\sum_{i=0}^nx_{[0,i]}\otimes x_{[i,n]}.\]For $x,y\in Y_{\infty}$ with $\deg(x)=m$, a \emph{hash product} $x\hash{}y$ is determined by a partition $P$ of $y$ and an order preserving map from $P$ to the (numbering of the) leaves of $x$, that is $f:P\to[m]_0$. So, $x\hash{}y$ is the binary tree obtained by gluing the blocks of $P$ on the leaves of $x$ respect to $f$. Hence, the \emph{product} of $x$ with $y$ is\[x*y=\sum x\hash{}y.\]

The tuple $(\field[\T_{\infty}],*_{op},\Delta)$ is precisely the Hopf algebra structure described in \cite{AgSo04}. This implies that the dual graded Hopf algebra of $\field[\T_{\infty}]$ is isomorphic to the Loday--Ronco Hopf algebra \cite{LoRo06}.

\subsubsection*{Acknowledgements}

This work is part of the research group GEMA Res.180/2019 VRI--UA. The first author was supported, in part, by the the grant Fondo de Apoyo a la Investigaci\'on DIUA179-2020.

\bibliographystyle{plain}
\bibliography{bibtex.bib}

\listoffigures


$\,$\\[1cm]
Number of words in the document: 11271.

\end{document}